\documentclass{amsart}%
\usepackage{amsmath}
\usepackage{amsfonts}
\usepackage{amssymb}
\usepackage{graphicx}
\usepackage{color}
\usepackage{subfig}
\usepackage{caption}%
\providecommand{\U}[1]{\protect\rule{.1in}{.1in}}
\newtheorem{theorem}{Theorem}[section]

\newtheorem{corollary}[theorem]{Corollary}

\newtheorem{lemma}[theorem]{Lemma}

\newtheorem{proposition}[theorem]{Proposition}
\theoremstyle{definition}
\newtheorem{definition}[theorem]{Definition}
\newtheorem{example}[theorem]{Example}

\theoremstyle{remark}
\newtheorem{remark}[theorem]{Remark}
\begin{document}
\title{A general framework for island systems}
\author[S.~Foldes]{Stephan Foldes}
\address[Stephan Foldes]{Tampere University of Technology, PL 553, 33101 Tampere, Finland}
\email{stephan.foldes@tut.fi}
\author[E.~K.~Horv\'{a}th]{Eszter K. Horv\'{a}th}
\address[Eszter K. Horv\'{a}th]{Bolyai Institute, University of Szeged, Aradi
v\'{e}rtan\'{u}k tere 1, H-6720 Szeged, Hungary}
\email{horeszt@math.u-szeged.hu}
\author[S.~Radeleczki]{S\'{a}ndor Radeleczki}
\address[S\'{a}ndor Radeleczki]{Institute of Mathematics, University of Miskolc, 3515
Miskolc-Egyetemv\'{a}ros, Hungary}
\email{matradi@uni-miskolc.hu}
\author[T.~Waldhauser]{Tam\'{a}s Waldhauser}
\address[Tam\'{a}s Waldhauser]{Bolyai Institute, University of Szeged, Aradi
v\'{e}rtan\'{u}k tere 1, H-6720 Szeged, Hungary}
\email{twaldha@math.u-szeged.hu}

\begin{abstract}
The notion of an island defined on a rectangular board is an elementary
combinatorial concept that occurred first in \cite{czedli}. Results of
\cite{czedli} were starting points for investigations exploring several
variations and various aspects of this notion.

In this paper we introduce a general framework for islands that subsumes all
earlier studied concepts of islands on finite boards, moreover we show that
the prime implicants of a Boolean function, the formal concepts of a formal
context, convex subgraphs of a simple graph, and some particular subsets of
a projective plane also fit into this framework.

We axiomatize those cases where islands have the property of being pairwise comparable or disjoint, or they are distant, introducing the notion of a connective island
domain and of a proximity domain, respectively. In the general case the
maximal systems of islands are characterised by using the concept of an
admissible system. We also characterise all possible island systems in the
case of connective island domains and proximity domains.

\end{abstract}
\maketitle

\section{Introduction\label{section introduction}}

\textquotedblleft ISLAND, in physical geography, a term generally definable as
a piece of land surrounded by water.\textquotedblright\ (Encyclop\ae dia
Britannica, Eleventh Edition, Volume XIV, Cambridge University Press 1910.)
Mathematical models of this definition were introduced and studied by several
authors. These investigations utilized tools from different areas of
mathematics, e.g. combinatorics, coding theory, lattice theory, analysis,
fuzzy mathematics. Our goal is to provide a general setting that unifies these
approaches. This general framework encompasses prime implicants of Boolean
functions and concepts of a formal context as special cases, and it has close
connections to graph theory and to proximity spaces.

The notion of an island as a mathematical concept occurred first in Cz\'{e}dli
\cite{czedli}, where a rectangular board was considered with a real number
assigned to each cell of the board, representing the height of that cell. A
set $S$ of cells forming a rectangle is called an \it island, \rm if the minimum
height of $S$ is greater then the height of any cell around the perimeter of
$S$, since in this case $S$ can become a piece of land surrounded by water
after a flood producing an appropriate water level. The motivation to
investigate such islands comes from Foldes and Singhi \cite{FS}, where islands
on a $1\times n$ board (so-called full segments) played a key role in
characterizing maximal instantaneous codes.

The main result of \cite{czedli} is that the maximum number of islands on an
$m\times n$ board is $\left\lfloor \left(  mn+m+n-1\right)  /2\right\rfloor $.
However, the size of a system of islands (i.e., the collection of all islands appearing for given heights) that is maximal with respect to
inclusion (not with respect to cardinality) can be as low as $m+n-1$ \cite{L}.
Another important observation of \cite{czedli} is that any two islands are
either comparable (i.e. one is contained in the other) or disjoint; moreover,
disjoint islands cannot be too close to each other (i.e. they cannot have
neighboring cells). It was also shown in \cite{czedli} that these properties
actually characterize systems of islands. We refer to such a result as a
\textquotedblleft dry\textquotedblright\ characterization, since it describes
systems of islands in terms of intrinsic conditions, without referring to
heights and water levels.

The above mentioned paper \cite{czedli} of G\'{a}bor Cz\'{e}dli was a starting
point for many investigations exploring several variations and various aspects
of islands. Square islands on a rectangular board have been considered in
\cite{HHNS,L3}, and islands have been studied also on cylindrical and toroidal
boards \cite{BHH}, on triangular boards \cite{HNP,L2}, on higher dimensional
rectangular boards \cite{P} as well as in a continuous setting \cite{LP,PPPSz}%
. If we allow only a given finite subset of the reals as possible heights,
then the problem of determining the maximum number of islands becomes
considerably more difficult; see, e.g. \cite{HTM,HST2, MM}. Islands also
appear naturally as cuts of lattice-valued functions \cite{HST}; furthermore,
order-theoretic properties of systems of islands proved to be of interest on
their own, and they have been investigated in lattices and partially ordered
sets \cite{CzHaSch,CzSch,HR}. The notion of an island is an elementary
combinatorial concept, yet it leads immediately to open problems, therefore it
is a suitable topic to introduce students to mathematical research \cite{MV}.

In this paper we introduce a general framework for islands that subsumes all
of the earlier studied concepts of islands on finite boards. We will
axiomatize those situations where islands have the \textquotedblleft
comparable or disjoint\textquotedblright\ property mentioned above, and we
will also present dry characterizations of systems of islands.

\section{Definitions and examples}

Our landscape is given by a nonempty base set $U$, and a function $h\colon
U\rightarrow\mathbb{R}$ that assigns to each point $u\in U$ its height
$h\left(  u\right)  $. If the minimum height $\min h\left(  S\right)
:=\min\left\{  h\left(  u\right)  \colon u\in S\right\}  $ of a set
$S\subseteq U$ is greater than the height of its surroundings, then $S$ can
become an island if the water level is just below $\min h\left(  S\right)  $.
To make this more precise, let us fix two families of sets $\mathcal{C}%
,\mathcal{K}\subseteq\mathcal{P}\left(  U\right)  $, where $\mathcal{P}\left(
U\right)  $ denotes the power set of $U$. We do not allow islands of arbitrary
\textquotedblleft shapes\textquotedblright: only sets belonging to
$\mathcal{C}$ are considered as candidates for being islands, and the members
of $\mathcal{K}$ describe the \textquotedblleft surroundings\textquotedblright%
\ of these sets.

\begin{definition}
An \emph{island domain} is a pair $\left(  \mathcal{C},\mathcal{K}\right)  $,
where $\mathcal{C}\subseteq\mathcal{K}\subseteq\mathcal{P}\left(  U\right)  $
for some nonempty finite set $U$ such that $U\in\mathcal{C}$. By a
\emph{height function} we mean a map $h\colon U\rightarrow\mathbb{R}$.
\end{definition}

Throughout the paper we will always implicitly assume that $\left(  \mathcal{C}%
,\mathcal{K}\right)  $ is an island domain. We denote the cover relation of
the poset $\left(  \mathcal{K},\subseteq\right)  $ by $\prec$, and we write
$K_{1}\preceq K_{2}$ if $K_{1}\prec K_{2}$ or $K_{1}=K_{2}$.

\begin{definition}
Let $\left(  \mathcal{C},\mathcal{K}\right)  $ be an island domain, let
$h\colon U\rightarrow\mathbb{R}$ be a height function and let $S\in
\mathcal{C}$ be a nonempty set.%
\renewcommand{\theenumi}{(\roman{enumi})}
\renewcommand{\labelenumi}{\theenumi}%

\begin{enumerate}
\item \label{def island (i)}We say that $S$ is a\emph{ pre-island} with
respect to the triple $\left(  \mathcal{C},\mathcal{K},h\right)  $, if every
$K\in\mathcal{K}$ with $S\prec K$ satisfies%
\[
\min h\left(  K\right)  <\min h\left(  S\right)  .
\]

\item \label{def island (ii)}We say that $S$ is an \emph{island} with respect
to the triple $\left(  \mathcal{C},\mathcal{K},h\right)  $, if every
$K\in\mathcal{K}$ with $S\prec K$ satisfies
\[
h\left(  u\right)  <\min h\left(  S\right)  \text{ for all }u\in K\setminus
S.
\]

\end{enumerate}

The \emph{system of (pre-)islands corresponding to }$\left(  \mathcal{C}%
,\mathcal{K},h\right)  $ is the set%
\[
\left\{  S\in\mathcal{C}\setminus\left\{  \emptyset\right\}  \colon S\text{ is
a (pre-)island w.r.t. }\left(  \mathcal{C},\mathcal{K},h\right)  \right\}  .
\]
By a \emph{system of (pre-)islands corresponding to }$\left(  \mathcal{C}%
,\mathcal{K}\right)  $ we mean a set $\mathcal{S}\subseteq\mathcal{C}$ such
that there is a height function $h\colon U\rightarrow\mathbb{R}$ so that the
system of (pre-)islands corresponding to $\left(  \mathcal{C},\mathcal{K}%
,h\right)  $ is $\mathcal{S}$.
\end{definition}

\begin{remark}
\label{remark basic}Let us make some simple observations concerning the above
definition.%
\renewcommand{\theenumi}{(\alph{enumi})}
\renewcommand{\labelenumi}{\theenumi}%

\begin{enumerate}
\item Every nonempty set $S$ in $\mathcal{C}$ is in fact an island for some height function  $h.$

\item If $S$ is an island with respect to $\left(  \mathcal{C},\mathcal{K}%
,h\right)  $, then $S$ is also a pre-island with respect to $\left(
\mathcal{C},\mathcal{K},h\right)  $. The converse is not true in general;
however, if for every nonempty\emph{ }$C\in\mathcal{C}$ and $K\in\mathcal{K}$
with\emph{ }$C\prec K$ we have $\left\vert K\setminus C\right\vert =1$, then
the two notions coincide.

\item The set $U$ is always a (pre-)island. If $S$ is a (pre-)island that is
different from $U$, then we say that $S$ is a \emph{proper (pre-)island}.

\item If $S$ is a pre-island with respect to $\left(  \mathcal{C}%
,\mathcal{K},h\right)  $, then the inequality $\min h\left(  K\right)  <\min
h\left(  S\right)  $ of \ref{def island (i)} holds for all $K\in\mathcal{K}$
with $S\subset K$ (not just for covers of $S)$.

\item Let $\mathcal{C}\subseteq\mathcal{K}^{\prime}\subseteq\mathcal{K}$. It
is easy to see that any $\mathcal{S}\in\mathcal{C}$ which is a pre-island with
respect to the triple $\left(  \mathcal{C},\mathcal{K},h\right)  $ is also a
pre-island with respect to $\left(  \mathcal{C},\mathcal{K}^{\prime},h\right)
$.

\item The numerical values of the height function $h$ are not important; only
the partial ordering that $h$ establishes on $U$ is relevant. In particular,
one could assume without loss of generality that the range of $h$ is contained
in the set $\left\{  0,1,\ldots,\left\vert U\right\vert -1\right\}  $.
\end{enumerate}
\end{remark}

Many of the previously studied island concepts can be interpreted in terms of
graphs as follows.

\begin{example}
\label{example graphs}Let $G=\left(  U,E\right)  $ be a connected simple graph
with vertex set $U$ and edge set $E$; let $\mathcal{K}$ consist of the
connected subsets of $U$, and let $\mathcal{C}\subseteq\mathcal{K}$ such that
$U\in\mathcal{C}$. In this case the second item of Remark~\ref{remark basic}
applies, hence pre-islands and islands are the same. Let us assume that $G$ is
connected, and let $\mathcal{C}$ consist of the connected convex sets of
vertices. (A set is called convex if it contains all shortest paths between
any two of its vertices.) If $G$ is a path, then the islands are exactly the
full segments considered in \cite{FS}, and if $G$ is a square grid (the
product of two paths), then we obtain the rectangular islands of
\cite{czedli}. Square islands on a rectangular board \cite{HHNS,L3}, islands
on cylindrical and toroidal boards \cite{BHH}, on triangular boards
\cite{HNP,L2} and on higher dimensional rectangular boards \cite{P} also fit
into this setting.
\end{example}

Surprisingly, formal concepts and prime implicants are also pre-islands in disguise.

\begin{example}
\label{example context}Let $A_{1},\ldots,A_{n}$ be nonempty sets, and let
$\mathcal{I}\subseteq A_{1}\times\cdots\times A_{n}$. Let us define%
\begin{align*}
U  &  =A_{1}\times\cdots\times A_{n},\\
\mathcal{K}  &  =\left\{  B_{1}\times\cdots\times B_{n}\colon\emptyset\neq
B_{i}\subseteq A_{i},~1\leq i\leq n\right\} \\
\mathcal{C}  &  =\left\{  C\in\mathcal{K}\colon C\subseteq\mathcal{I}\right\}
\cup\{U\},
\end{align*}
and let $h\colon U\longrightarrow\{0,1\}$ be the height function given by%
\[
h\left(  a_{1},\ldots,a_{n}\right)  :=\left\{  \!\!%
\begin{array}
[c]{rl}%
1\text{,} & \text{if }\left(  a_{1},\ldots,a_{n}\right)  \in\mathcal{I};\\
0\text{,} & \text{if }\left(  a_{1},\ldots,a_{n}\right)  \in U\setminus
\mathcal{I};
\end{array}
\right.  \text{ for all }\left(  a_{1},\ldots,a_{n}\right)  \in U.
\]
It is easy to see that the pre-islands corresponding to the triple $\left(
\mathcal{C},\mathcal{K},h\right)  $ are exactly $U$ and the maximal elements
of the poset $\left(  \mathcal{C}\setminus\left\{  U\right\}  ,\subseteq
\right)  $.

Now let $\left(  G,M,\mathcal{I}\right)  $, $\mathcal{I}\subseteq G\times M$
be a formal context, and let us apply the above construction with $A_{1}=G$,
$A_{2}=M$ and $U=A_{1}\times A_{2}$. Then the pre-islands are $U$ and the
concepts of the context $(G,M,\mathcal{I)}$ with nonempty extent and intent
\cite{GW}.

Further, consider the case $A_{1}=\cdots=A_{n}=\{0,1\}$. Then the height
function $h$ is an $n$-ary Boolean function, and it is not hard to check that
the pre-islands corresponding to $\left(  \mathcal{C},\mathcal{K},h\right)  $
are $U$ and the prime implicants of $h$ \cite{CrHa}.
\end{example}

\begin{remark}
\label{rem refinement}For any given island domain $\left(  \mathcal{C}%
,\mathcal{K}\right)  $, maximal families of (pre-)islands are realized by
injective height functions. To see this, let us assume that $h$ is a
non-injective height function, i.e. there exists a number $z$ in the range of
$h$ such that $h^{-1}\left(  z\right)  =\left\{  s_{1},\ldots,s_{m}\right\}  $
with $m\geq2$. The following \textquotedblleft refinement\textquotedblright%
\ procedure constructs another height function $g$ so that every (pre-)island
corresponding to $\left(  \mathcal{C},\mathcal{K},h\right)  $ is also a
(pre-)island with respect to $\left(  \mathcal{C},\mathcal{K},g\right)  $. Let
$y$ be the largest value of $h$ below $z$ (or $z-1$ if $z$ is the minimum
value of the range of $h$), and let $w$ be the smallest value of $h$ above $z$
(or $z+1$ if $z$ is the maximum value of the range of $h$). For any $u\in U$,
we define $g\left(  u\right)  $ by%
\[
g\left(  u\right)  =\left\{  \!\!\renewcommand{\arraystretch}{1.8}%
\begin{array}
[c]{rl}%
y+i\dfrac{w-y}{m+1}, & \text{if }u=s_{i};\\
h\left(  u\right)  , & \text{if }h\left(  u\right)  \neq z.
\end{array}
\right.
\]
By repeatedly applying this procedure we obtain an injective height function
without losing any pre-islands. Note that injective height functions
correspond to linear orderings of $U$ (cf. the last observation of
Remark~\ref{remark basic}).
\end{remark}

\begin{example}
\label{example projective plane}Let $U$ be a finite projective plane of order
$p$, thus $U$ has $m:=p^{2}+p+1$ points. Let $\mathcal{C}=\mathcal{K}$ consist
of the whole plane, the lines, the points and the empty set. Then the greatest
possible number of pre-islands is $p^{2}+2=m-p+1$. Indeed, as explained in
Remark~\ref{rem refinement}, the largest systems of pre-islands emerge with
respect to linear orderings of $U$. So let us consider a linear order on $U$,
and let $\mathbf{0}$ and $\mathbf{1}$ denote the smallest and largest elements
of $U$, respectively. In other words, we have $h\left(  \mathbf{0}\right)
<h(x)<h\left(  \mathbf{1}\right)  $ for all $x\in U\setminus\{\mathbf{0}%
,\mathbf{1}\}$. Clearly, a line is a pre-island iff it does not contain
$\mathbf{0}$, and there are $m-p-1$ such lines. The only other pre-islands are
the point $\mathbf{1}$ and the entire plane, hence we obtain $m-p-1+2=m-p+1$ pre-islands.
\end{example}

It has been observed in \cite{czedli, HNP,HHNS} that any two islands on a
square or triangular grid with respect to a given height function are either
comparable or disjoint. This property is formalized in the following
definition, which was introduced in \cite{CzHaSch}.

\begin{definition}
A family $\mathcal{H}$ of subsets of $U$ is $\operatorname{CD}$%
\emph{-independent} if any two members of $\mathcal{H}$ are either comparable
or disjoint, i.e. for all $A,B\in\mathcal{H}$ at least one of $A\subseteq B$,
$B\subseteq A$ or $A\cap B=\emptyset$ holds.
\end{definition}

Note that $\operatorname{CD}$-independence is also known as laminarity
\cite{LP,PPPSz}. In general, the properties of $\operatorname{CD}%
$-independence and being a system of pre-islands are independent from each
other, as the following example shows.

\begin{example}
\label{example not CD}Let $U=\left\{  a,b,c,d,e\right\}  $ and $\mathcal{K}%
=\mathcal{C}=\left\{  \left\{  a,b\right\}  ,\left\{  a,c\right\}  ,\left\{
b,d\right\}  ,\left\{  c,d\right\}  ,U\right\}  $. Let us define a height
function $h$ on $U$ by $h\left(  a\right)  =h\left(  b\right)  =h\left(
c\right)  =h\left(  d\right)  =1$, $h\left(  e\right)  =0$. It is easy to
verify that every element of $\mathcal{C}$ is a pre-island with respect to
this height function, but $\mathcal{C}$ is not $\operatorname{CD}%
$-independent. On the other hand, consider the $\operatorname{CD}$-independent
family $\mathcal{H}=\left\{  \left\{  a,b\right\}  ,\left\{  c,d\right\}
,U\right\}  $. We claim that $\mathcal{H}$ is not a system of pre-islands. To
see this, assume that $h$ is a height function such that the system of
pre-islands corresponding to $\left(  \mathcal{C},\mathcal{K},h\right)  $ is
$\mathcal{H}$. Let us write out the definition of a pre-island for $S=\left\{
a,b\right\}  $ and $S=\left\{  c,d\right\}  $ with $K=U$:%
\begin{align*}
\min\left(  h\left(  a\right)  ,h\left(  b\right)  \right)   &  >\min h\left(
U\right)  ;\\
\min\left(  h\left(  c\right)  ,h\left(  d\right)  \right)   &  >\min h\left(
U\right)  .
\end{align*}
Taking the minimum of these two inequalities, we obtain%
\[
\min\left(  h\left(  a\right)  ,h\left(  b\right)  ,h\left(  c\right)
,h\left(  d\right)  \right)  >\min h\left(  U\right)  .
\]
This immediately implies that $\min\left(  h\left(  a\right)  ,h\left(
c\right)  \right)  >\min h\left(  U\right)  $. Since the only element of
$\mathcal{K}$ properly containing $\left\{  a,c\right\}  $ is $U$, we can
conclude that $\left\{  a,c\right\}  $ is also a pre-island with respect to
$h$, although $\left\{  a,c\right\}  \notin\mathcal{H}$.
\end{example}

As $\operatorname{CD}$-independence is a natural and desirable property of
islands that was crucial in previous investigations, we will mainly focus on
island domains $\left(  \mathcal{C},\mathcal{K}\right)  $ whose systems of
pre-islands are $\operatorname{CD}$-independent. We characterize such island
domains in Theorem~\ref{thm ID <==> (SZ==>CD)}, and we refer to them as
\emph{connective island domains} (see Definition~\ref{def ID}).

The most fundamental questions concerning pre-islands are the following: Given
an island domain $\left(  \mathcal{C},\mathcal{K}\right)  $ and a family
$\mathcal{H}\subseteq\mathcal{C}$, how can we decide if there is a height
function $h$ such that $\mathcal{H}$ is the system of pre-islands
corresponding to $\left(  \mathcal{C},\mathcal{K},h\right)  $? How can we find
such a height function (if there is one)? Concerning the first question, we
give a dry\ characterization of systems of pre-islands corresponding to
connective island domains in Theorem~\ref{thm ID ==> (HA<==>SZ)}, and in
Corollary~\ref{cor PD ==> (distant<==>strSZ)} we characterize systems of
islands corresponding to so-called \emph{proximity domains} (see
Definition~\ref{def proximity domain}). These results generalize earlier dry
characterizations (see, e.g. \cite{czedli, HNP,HHNS}), since an island domain
$\left(  \mathcal{C},\mathcal{K}\right)  $ corresponding to a graph (cf.
Example~\ref{example graphs}) is always a connective island domain and also a
proximity domain. Concerning the second question, we give a canonical
construction for a height function
(Definition~\ref{def canonical height function}), and we prove in
Sections~\ref{section CD and ID} and \ref{section strict} that this height
function works for pre-islands in connective island domains and for islands in
proximity domains.

\section{Pre-islands and admissible systems\label{section admissible}}

In this section we present a condition that is necessary for being a system of
pre-islands, which will play a key role in later sections. Although this
necessary condition is not sufficient in general, we will use it to obtain a
characterization of \emph{maximal} systems of pre-islands.

\begin{definition}
\label{def admissible}Let $\mathcal{H}\subseteq\mathcal{C}\setminus\left\{
\emptyset\right\}  $ be a family of sets such that $U\in\mathcal{H}$. We say
that $\mathcal{H}$ is \emph{admissible (with respect to }$\left(
\mathcal{C},\mathcal{K}\right)  $\emph{)}, if for every nonempty antichain
$\mathcal{A}\subseteq\mathcal{H}$,%
\begin{equation}
\exists H\in\mathcal{A}\text{ such that }\forall K\in\mathcal{K}:~H\subset
K\implies K\nsubseteq\bigcup\,\mathcal{A}.\label{eq admissible}%
\end{equation}

\end{definition}

\begin{remark}
\label{rem admissible not just for antichains}Let us note that if
$\mathcal{H}$ is admissible, then (\ref{eq admissible}) holds for \emph{all}
nonempty $\mathcal{A}\subseteq\mathcal{H}$ (not just for antichains). Indeed,
if $\mathcal{M}$ denotes the set of maximal members of $\mathcal{A}$, then
$\mathcal{M}$ is an antichain. Thus the admissibility of $\mathcal{H}$ implies
that there is $H\in\mathcal{M}\subseteq\mathcal{A}$ such that for all
$K\in\mathcal{K}$ with $H\subset K$ we have $K\nsubseteq\bigcup\,\mathcal{M}%
=\bigcup\,\mathcal{A}$.
\end{remark}

Obviously, any subfamily of an admissible family is also admissible, provided
that it contains $U$. As we shall see later, in some important special cases a
stronger version of admissibility holds, where the existential quantifier is
replaced by a universal quantifier in (\ref{eq admissible}): for every
nonempty antichain $\mathcal{A}\subseteq\mathcal{H}$,%
\begin{equation}
\forall H\in\mathcal{A~}\forall K\in\mathcal{K}:~H\subset K\implies
K\nsubseteq\bigcup\,\mathcal{A}. \label{eq stronger admissible}%
\end{equation}

\begin{proposition}
\label{prop SZ==>HA}Every system of pre-islands is admissible.
\end{proposition}

\begin{proof}
Let $h\colon U\rightarrow\mathbb{R}$ be a height function and let
$\mathcal{S}$ be the system of pre-islands corresponding to $\left(
\mathcal{C},\mathcal{K},h\right)  $. Clearly, we have $\emptyset
\notin\mathcal{S}$ and $U\in\mathcal{S}$. Let us assume for contradiction that
there exists an antichain $\mathcal{A}=\left\{  S_{i}:i\in I\right\}
\subseteq\mathcal{S}$ such that (\ref{eq admissible}) does not hold. Then for
every $i\in I$ there exists $K_{i}\in\mathcal{K}$ such that $S_{i}\subset
K_{i}$ and $K_{i}\subseteq\bigcup_{i\in I}S_{i}$. Since $S_{i}$ is a
pre-island, we have%
\[
\min h\left(  S_{i}\right)  >\min h\left(  K_{i}\right)  \geq\min h%
\Bigl(%
\bigcup_{i\in I}S_{i}%
\Bigr)%
\]
for all $i\in I$. Taking the minimum of these inequalities we arrive at the
contradiction%
\[
\min\left\{  \min h\left(  S_{i}\right)  \mid i\in I\right\}  >\min h%
\Bigl(%
\bigcup_{i\in I}S_{i}%
\Bigr)%
.%
\qedhere
\]

\end{proof}

The converse of Proposition~\ref{prop SZ==>HA} is not true in general: it is
straightforward to verify that the family $\mathcal{H}$ considered in
Example~\ref{example not CD} is admissible, but, as we have seen, it is not a
system of pre-islands. However, we will prove in
Proposition~\ref{prop HA==>SZ resze} that for every admissible family
$\mathcal{H}$, there exists a height function such that the corresponding
system of pre-islands contains $\mathcal{H}$. First we give the construction
of this height function, and we illustrate it with some examples.

\begin{definition}
\label{def canonical height function}Let $\mathcal{H}\subseteq\mathcal{C}$ be
an admissible family of sets. We define subfamilies $\mathcal{H}^{\left(
i\right)  }\subseteq\mathcal{H}~\left(  i=0,1,2,\ldots\right)  $ recursively
as follows. Let $\mathcal{H}^{\left(  0\right)  }=\left\{  U\right\}  $. For
$i>0$, if $\mathcal{H}\neq\mathcal{H}^{\left(  0\right)  }\cup\cdots
\cup\mathcal{H}^{\left(  i-1\right)  }$, then let $\mathcal{H}^{\left(
i\right)  }$ consist of all those sets $H\in\mathcal{H}\setminus
(\mathcal{H}^{\left(  0\right)  }\cup\cdots\cup\mathcal{H}^{\left(
i-1\right)  })$ that have the following property:%
\begin{equation}
\forall K\in\mathcal{K}:~H\subset K\implies K\nsubseteq\bigcup\,%
\bigl(%
\mathcal{H}\setminus(\mathcal{H}^{\left(  0\right)  }\cup\cdots\cup
\mathcal{H}^{\left(  i-1\right)  })%
\bigr)%
. \label{eq def canonical height function: set}%
\end{equation}
Since $\mathcal{H}$ is finite and admissible, after finitely many steps we
obtain a partition $\mathcal{H}=\mathcal{H}^{\left(  0\right)  }\cup\cdots
\cup\mathcal{H}^{\left(  r\right)  }$ (cf.
Remark~\ref{rem admissible not just for antichains}). The \emph{canonical
height function corresponding to }$\mathcal{H}$ is the function
$h_{\mathcal{H}}\colon U\rightarrow\mathbb{N}$ defined by%
\begin{equation}
h_{\mathcal{H}}\left(  x\right)  :=\max\left\{  i\in\left\{  1,\ldots
,r\right\}  :x\in\bigcup\,\mathcal{H}^{\left(  i\right)  }\right\}  \text{ for
all }x\in U. \label{eq def canonical height function: height}%
\end{equation}

\end{definition}

Observe that every $\mathcal{H}^{\left(  i\right)  }$ consists of \emph{some}
of the maximal members of $\mathcal{H}\setminus(\mathcal{H}^{\left(  0\right)
}\cup\cdots\cup\mathcal{H}^{\left(  i-1\right)  })=\mathcal{H}^{\left(
i\right)  }\cup\cdots\cup\mathcal{H}^{\left(  r\right)  }$. However, if
$\mathcal{H}$ satisfies (\ref{eq stronger admissible}) for all antichains
$\mathcal{A}\subseteq\mathcal{H}$, then the word \textquotedblleft%
\emph{some}\textquotedblright\ can be replaced by \textquotedblleft%
\emph{all}\textquotedblright\ in the previous sentence, and in this case
$h_{\mathcal{H}}$ can be computed just from $\mathcal{H}$ itself, without
making reference to $\mathcal{K}$. To illustrate this, let us consider a
$\operatorname{CD}$-independent family $\mathcal{H}$. Clearly, for every $u\in
U$, the set of members of $\mathcal{H}$ containing $u$ is a finite chain. The
\emph{standard height function} of $\mathcal{H}$ assigns to each element $u$
the length of this chain, i.e. one less than the number of members of
$\mathcal{H}$ that contain $u$. (Note that the definition of a standard height
function in \cite{HST2} differs slightly from ours.) It is easy to see that if
$\mathcal{H}$ satisfies (\ref{eq stronger admissible}), then the canonical
height function of $h$ coincides with the standard height function. However,
in general the two functions might be different.
Figure~\ref{fig standardeskanonikus} represents the standard and the canonical
height functions for the same $\operatorname{CD}$-independent family, with
greater heights indicated by darker colors. We can see from
Figure~\ref{fig kanonikus} that only two of the four maximal members of
$\mathcal{H}\setminus\left\{  U\right\}  $ belong to $\mathcal{H}^{\left(
1\right)  }$, thus (\ref{eq stronger admissible}) fails here. (In order to
make the picture comprehensible, only members of $\mathcal{C}$ are shown,
although $\mathcal{K}$ is also needed to determine $h_{\mathcal{H}}$
(Figure~\ref{fig kanonikus}). On the other hand, the standard height function
(Figure~\ref{fig standard}) can be read directly from the figure.)

\begin{figure}[h]
\centering
\subfloat[Standard height function] {
\includegraphics[width=0.4\textwidth]{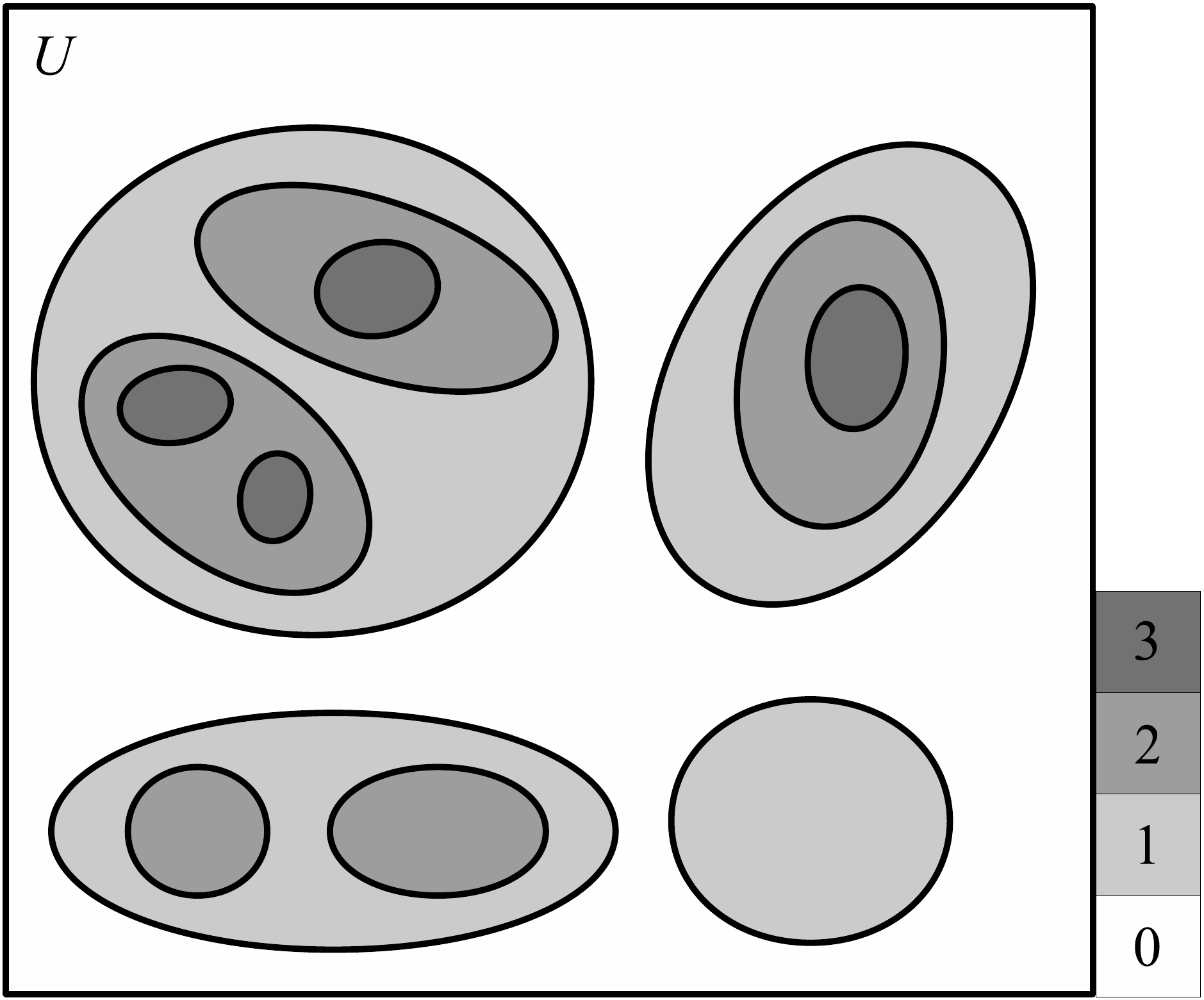} \label{fig standard}}
\qquad\subfloat[Canonical height function]{
\includegraphics[width=0.4\textwidth]{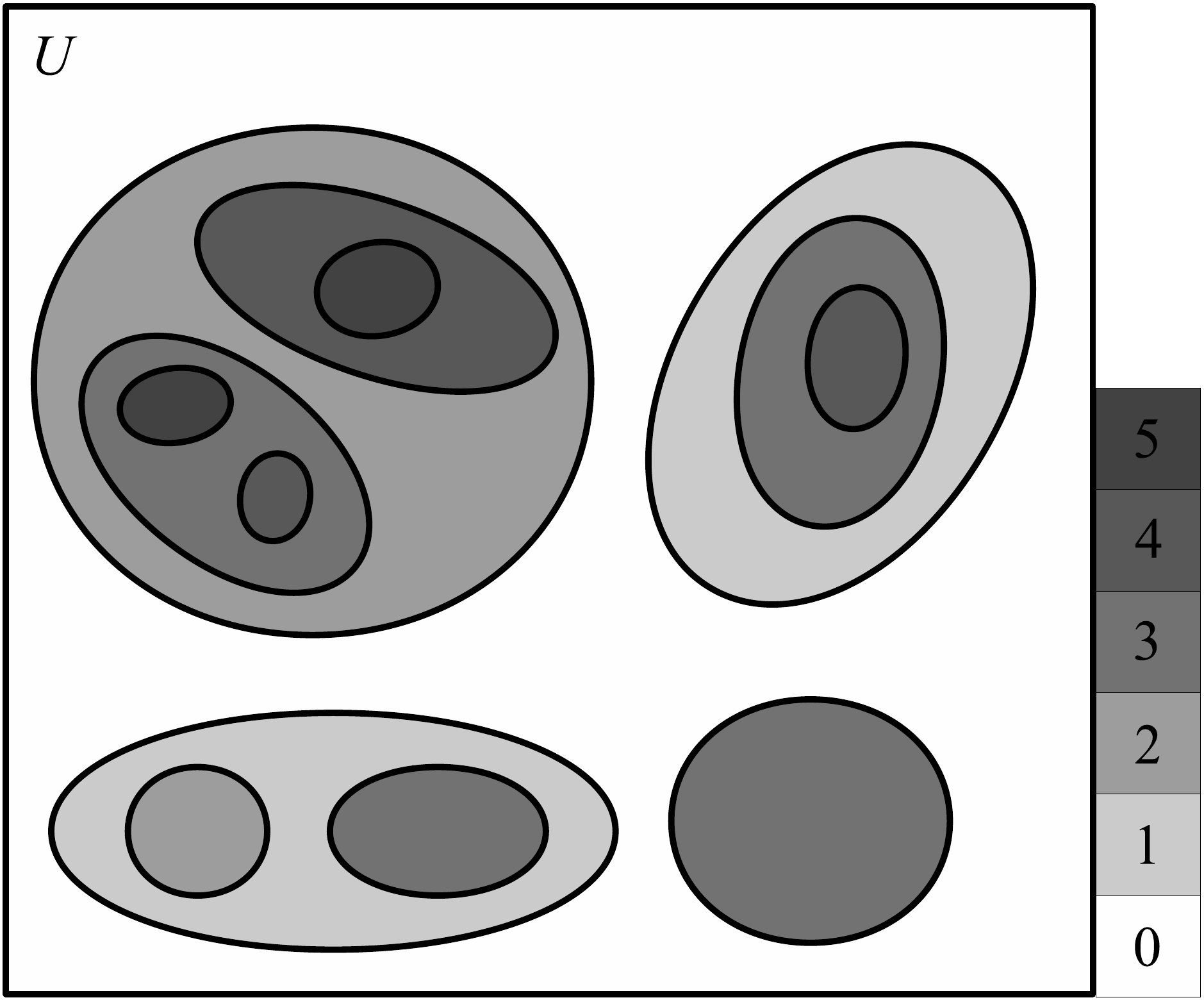}
\label{fig kanonikus}}\caption{A $\operatorname{CD}$-independent family with
two different height functions}%
\label{fig standardeskanonikus}%
\end{figure}

The next example shows that there exist $\operatorname{CD}$-independent
systems of pre-islands for which the standard height function is not the right
choice. However, in Section~\ref{section strict} we will see that for a wide
class of island domains, including those corresponding to graphs (cf.
Example~\ref{example graphs}), the standard height function is always appropriate.

\begin{example}
\label{example not standard}Let $U=\left\{  a,b,c,d\right\}  $, $\mathcal{C}%
=\left\{  A,B,U\right\}  $ and $\mathcal{K}=\left\{  A,B,U,K\right\}  $, where
$A=\left\{  a\right\}  $, $B=\left\{  b,c\right\}  $ and $K=\left\{
a,c\right\}  $. Then the family $\mathcal{H}=\left\{  A,B,U\right\}  $ is
admissible; the corresponding partition is $\mathcal{H}^{\left(  0\right)
}=\left\{  U\right\}  $, $\mathcal{H}^{\left(  1\right)  }=\left\{  B\right\}
$, $\mathcal{H}^{\left(  2\right)  }=\left\{  A\right\}  $, and the canonical
height function is given by $h_{\mathcal{H}}\left(  a\right)  =2$,
$h_{\mathcal{H}}\left(  b\right)  =h_{\mathcal{H}}\left(  c\right)  =1$,
$h_{\mathcal{H}}\left(  d\right)  =0$. It is straightforward to verify that
$\mathcal{H}$ is the system of pre-islands corresponding to $\left(
\mathcal{C},\mathcal{K},h_{\mathcal{H}}\right)  $. However, the standard
height function assigns the value $1$ to $a$, and thus $A$ is not a pre-island
with respect to the standard height function of $\mathcal{H}$.
\end{example}

\begin{proposition}
\label{prop HA==>SZ resze}If $\mathcal{H\subseteq C}$ is an admissible family
of sets and $h_{\mathcal{H}}$ is the corresponding canonical height function,
then every member of $\mathcal{H}$ is a pre-island with respect to $\left(
\mathcal{C},\mathcal{K},h_{\mathcal{H}}\right)  $.
\end{proposition}

\begin{proof}
Let $\mathcal{H\subseteq C}$ be admissible, and let us consider the partition
$\mathcal{H}=\mathcal{H}^{\left(  0\right)  }\cup\cdots\cup\mathcal{H}%
^{\left(  r\right)  }$ given in Definition~\ref{def canonical height function}%
. For each $H\in\mathcal{H}$, there is a unique $i\in\left\{  1,\ldots
,r\right\}  $ such that $H\in\mathcal{H}^{\left(  i\right)  }$, and we have
$\min h_{\mathcal{H}}\left(  H\right)  \geq i$ by
(\ref{eq def canonical height function: height}). Using this observation it is
straightforward to verify that $H$ is indeed a pre-island with respect to
$\left(  \mathcal{C},\mathcal{K},h_{\mathcal{H}}\right)  $.
\end{proof}

As an immediate consequence of Propositions~\ref{prop SZ==>HA} and
\ref{prop HA==>SZ resze} we have the following corollary.

\begin{corollary}
\label{cor maximal HA <==> maximal SZ}A subfamily of $\mathcal{C}$ is a
maximal system of pre-islands if and only if it is a maximal admissible family.
\end{corollary}

We have seen in Example~\ref{example not CD} that it is possible that a subset
of a system of pre-islands is not a system of pre-islands. The notion of
admissibility allows us to describe those situations where this cannot happen.

\begin{proposition}
\label{prop subsystem}The following two conditions are equivalent for any
island domain $\left(  \mathcal{C},\mathcal{K}\right)  $:%
\renewcommand{\theenumi}{(\roman{enumi})}
\renewcommand{\labelenumi}{\theenumi}%

\begin{enumerate}
\item \label{prop subsystem (i)}Any subset of a system of pre-islands
corresponding to $\left(  \mathcal{C},\mathcal{K}\right)  $ that contains $U$
is also a system of pre-islands.

\item \label{prop subsystem (ii)}The systems of pre-islands corresponding to
$\left(  \mathcal{C},\mathcal{K}\right)  $ are exactly the admissible families.
\end{enumerate}
\end{proposition}

\begin{proof}
The implication \ref{prop subsystem (ii)}$\implies$\ref{prop subsystem (i)}
follows from the simple observation that any subset of an admissible family
containing $U$ is also admissible. Assume now that \ref{prop subsystem (i)}
holds. In view of Proposition~\ref{prop SZ==>HA}, it suffices to prove that
every admissible family is a system of pre-islands. Let $\mathcal{H}$ be an
admissible family, then Proposition~\ref{prop HA==>SZ resze} yields a system
of pre-islands containing $\mathcal{H}$. Using \ref{prop subsystem (i)} we can
conclude that $\mathcal{H}$ is a system of pre-islands.
\end{proof}

\section{$\operatorname{CD}$-independence and connective island
domains\label{section CD and ID}}

As we have seen in Example~\ref{example not CD}, a system of pre-islands is
not necessarily $\operatorname{CD}$-independent. In this section we present a
condition that characterizes those island domains $\left(  \mathcal{C}%
,\mathcal{K}\right)  $ whose systems of pre-islands are $\operatorname{CD}%
$-independent, and we will prove that admissibility is necessary and
sufficient for being a systems of pre-islands in this case.

\begin{definition}
\label{def ID}An island domain $\left(  \mathcal{C},\mathcal{K}\right)  $ is a
\emph{connective island domain} if%
\begin{equation}
\forall A,B\in\mathcal{C}:~\left(  A\cap B\neq\emptyset\text{ and }B\nsubseteq
A\right)  \implies\exists K\in\mathcal{K}:A\subset K\subseteq A\cup B.
\label{eq def ID}%
\end{equation}

\end{definition}

\begin{remark}
\label{rem ID symmetry}Observe that if $A\subset B$, then (\ref{eq def ID}) is
satisfied with $K=B$. Thus it suffices to require (\ref{eq def ID}) for sets
$A,B$ that are not comparable or disjoint. In this case, by switching the role
of $A$ and $B$, we obtain that there is also a set $K^{\prime}\in\mathcal{K}$
such that $B\subset K^{\prime}\subseteq A\cup B$ (see
Figure~\ref{fig islanddomain}).
\end{remark}

%

\begin{figure}
[tb]
\begin{center}
\includegraphics[
natheight=4.002400in,
natwidth=6.353600in,
height=4.0958cm,
width=6.4617cm
]%
{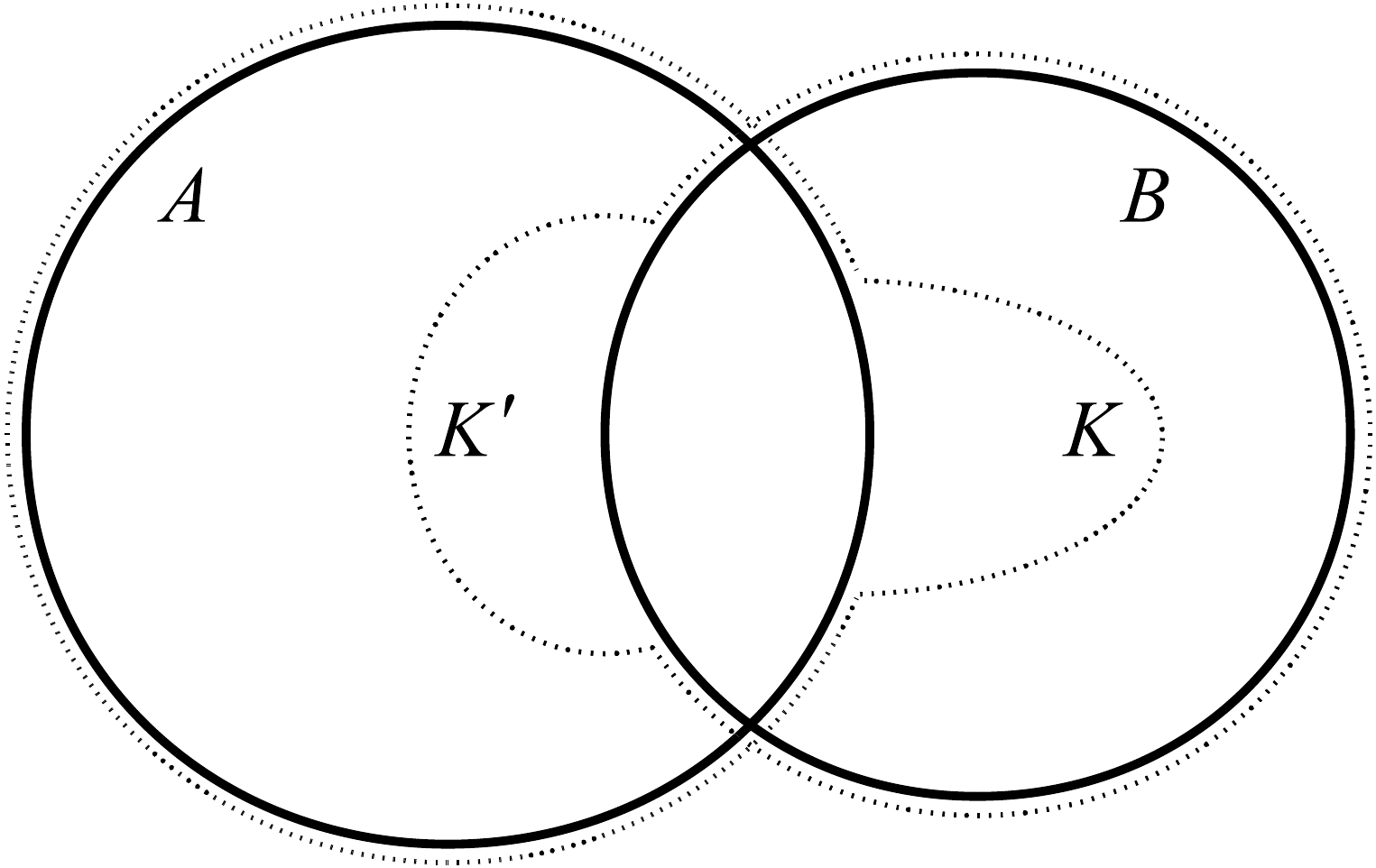}%
\caption{Illustration to the definition of an island domain}%
\label{fig islanddomain}%
\end{center}
\end{figure}

\begin{remark}
The terminology is motiveted by the intuition that the set $K$ in
Definition~\ref{def ID} somehow connects $A$ and $B$. Let us note that if
$\left(  \mathcal{C},\mathcal{K}\right)  $ corresponds to a graph, as in
Example~\ref{example graphs}, then $\left(  \mathcal{C},\mathcal{K}\right)  $
is a connective island domain. Furthermore, it is not difficult to prove that
if $\left(  \mathcal{C},\mathcal{K}\right)  $ is a connective island domain
with $\mathcal{C}=\mathcal{K}$, then (\ref{eq def ID}) is equivalent to the
fact that the union of two overlapping members of $\mathcal{K}$ belongs to
$\mathcal{K}$ (see (\ref{eq KK}) in\ Section~\ref{section strict}), which is
an important property of connected sets.
\end{remark}

We will prove that pre-islands corresponding to connective island domains are
not only $\operatorname{CD}$-independent, but they also satisfy the following
stronger independence condition, usually called $\operatorname*{CDW}%
$-independence, which was introduced in \cite{CzSch}.

\begin{definition}
\label{def CDW-independence}A family $\mathcal{H}\subseteq\mathcal{P}\left(
U\right)  $ is \emph{weakly independent} (see \cite{CzHSch}) if
\begin{equation}
H\subseteq\bigcup_{i\in I}H_{i}\implies\exists i\in I:H\subseteq H_{i}
\label{eq def CDW-independence}%
\end{equation}
holds for all $H\in\mathcal{H},H_{i}\in\mathcal{H}\left(  i\in I\right)  $. If
$\mathcal{H}$ is both $\operatorname{CD}$-independent and weakly independent,
then we say that $\mathcal{H}$ is $\operatorname*{CDW}$\emph{-independent}.
\end{definition}

\begin{remark}
\label{rem CD maximal subsets}Let $\mathcal{H}\subseteq\mathcal{P}\left(
U\right)  $ be a $\operatorname{CD}$-independent family, and let
$H\in\mathcal{H}$. Let $M_{1},\ldots,M_{m}$ be those elements of $\mathcal{H}$
that are properly contained in $H$ and are maximal with respect to this
property. Then $M_{1},\ldots,M_{m}$ are pairwise disjoint, and $M_{1}%
\cup\cdots\cup M_{m}\subseteq H$. Weak independence of $\mathcal{H}$ is
equivalent to the fact that this latter containment is strict for every
$H\in\mathcal{H}$. In particular, in the definition of weak independence it
suffices to require (\ref{eq def CDW-independence}) for pairwise disjoint sets
$H_{i}$.
\end{remark}

\begin{lemma}
\label{lemma ID ==> (HA==>CDW)}If $\left(  \mathcal{C},\mathcal{K}\right)  $
is a connective island domain, then every admissible subfamily of
$\mathcal{C}$ is $\operatorname*{CDW}$-independent.
\end{lemma}

\begin{proof}
Let $\left(  \mathcal{C},\mathcal{K}\right)  $ be a connective island domain,
and let $\mathcal{H}\subseteq\mathcal{C}$ be an admissible family. If
$A,B\in\mathcal{H}$ are neither comparable nor disjoint, then (\ref{eq def ID}%
) and Remark~\ref{rem ID symmetry} show that $\mathcal{A}:=\left\{
A,B\right\}  $ is an antichain for which (\ref{eq admissible}) does not hold
(see Figure~\ref{fig islanddomain}). Thus $\mathcal{H}$ is $\operatorname{CD}$-independent.

To prove that $\mathcal{H}$ is also $\operatorname*{CDW}$-independent, we
apply Remark~\ref{rem CD maximal subsets}. Let us assume for contradiction
that $M_{1}\cup\cdots\cup M_{m}=H$ for pairwise disjoint sets $M_{1}%
,\ldots,M_{m}\in\mathcal{H}\left(  m\geq2\right)  $ and $H\in\mathcal{H}$.
Since $M_{i}\subset H\in\mathcal{K}$ and $H\subseteq M_{1}\cup\cdots\cup
M_{m}$ for $i=1,\ldots,m$, we see that (\ref{eq admissible}) fails for the
antichain $\mathcal{A}:=\left\{  M_{1},\ldots,M_{m}\right\}  $, contradicting
the admissibility of $\mathcal{H}$.
\end{proof}

As the next example shows, a $\operatorname*{CDW}$-independent family in a
connective island domain is not necessarily admissible.

\begin{example}
\label{example CDNT but not HA}Let us consider the same sets $U$, $A$, $B$ and
$K$ as in Example$~$\ref{example not standard}, and let $\mathcal{C}=\left\{
A,B,U\right\}  $ and $\mathcal{K}=\left\{  A,B,U,K,L\right\}  $, where
$L=\left\{  a,b,c\right\}  $. Then $\left(  \mathcal{C},\mathcal{K}\right)  $
is a connective island domain and $\left\{  A,B,U\right\}  $ is
$\operatorname*{CDW}$-independent, but it is not admissible (hence not a
system of pre-islands).
\end{example}

\begin{theorem}
\label{thm ID <==> (SZ==>CD)}The following three conditions are equivalent for
any island domain $\left(  \mathcal{C},\mathcal{K}\right)  $:%
\renewcommand{\theenumi}{(\roman{enumi})}
\renewcommand{\labelenumi}{\theenumi}%

\begin{enumerate}
\item \label{thm ID <==> (SZ==>CD) (i)}$\left(  \mathcal{C},\mathcal{K}%
\right)  $ is a connective island domain.

\item \label{thm ID <==> (SZ==>CD) (ii)}Every system of pre-islands
corresponding to $\left(  \mathcal{C},\mathcal{K}\right)  $ is
$\operatorname{CD}$-independent.

\item \label{thm ID <==> (SZ==>CD) (iii)}Every system of pre-islands
corresponding to $\left(  \mathcal{C},\mathcal{K}\right)  $ is
$\operatorname*{CDW}$-independent.
\end{enumerate}
\end{theorem}

\begin{proof}
It is obvious that \ref{thm ID <==> (SZ==>CD) (iii)}$\implies$%
\ref{thm ID <==> (SZ==>CD) (ii)}.

To prove that \ref{thm ID <==> (SZ==>CD) (ii)}$\implies$%
\ref{thm ID <==> (SZ==>CD) (i)}, let us assume that $\left(  \mathcal{C}%
,\mathcal{K}\right)  $ is not a connective island domain. Then there exist
$A,B\in\mathcal{C}$ that are not comparable or disjoint such that there is no
$K\in\mathcal{K}$ with $A\subset K\subseteq A\cup B$. We define a height
function $h\colon U\rightarrow\mathbb{N}$ as follows:%
\[
h\left(  x\right)  :=\left\{  \!\!%
\begin{array}
[c]{rl}%
2, & \text{if }x\in B;\\
1, & \text{if }x\in A\setminus B;\\
0, & \text{if }x\notin A\cup B\text{.}%
\end{array}
\right.
\]
We claim that both $A$ and $B$ are pre-islands with respect to $\left(
\mathcal{C},\mathcal{K},h\right)  $. This is clear for $B$, as $\min h\left(
K\right)  \leq1$ for any proper superset $K$ of $B$. On the other hand, our
assumption implies that for any $K\supset A$ we have $K\nsubseteq A\cup B$,
hence $\min h\left(  K\right)  =0<\min h\left(  A\right)  =1$, thus $A$ is
indeed a pre-island. Since $A$ and $B$ are not $\operatorname{CD}$, the system
of pre-islands corresponding to $\left(  \mathcal{C},\mathcal{K},h\right)  $
is not $\operatorname{CD}$-independent.

Finally, for the implication \ref{thm ID <==> (SZ==>CD) (i)}$\implies
$\ref{thm ID <==> (SZ==>CD) (iii)}, assume that $\left(  \mathcal{C}%
,\mathcal{K}\right)  $ is a connective island domain and $\mathcal{S}$ is a
system of pre-islands corresponding to $\left(  \mathcal{C},\mathcal{K}%
\right)  $. By Proposition~\ref{prop SZ==>HA}, $\mathcal{S}$ is admissible,
and then Lemma~\ref{lemma ID ==> (HA==>CDW)} shows that $\mathcal{S}$ is
$\operatorname*{CDW}$-independent.
\end{proof}

Our final goal in this section is to prove that if $\left(  \mathcal{C}%
,\mathcal{K}\right)  $ is a connective island domain, then the systems of
pre-islands are exactly the admissible subfamilies of $\mathcal{C}$. Recall
that this is not true in general if $\left(  \mathcal{C},\mathcal{K}\right)  $
is not a connective island domain (see Example~\ref{example not CD}), but the
two notions coincide for maximal families
(Corollary~\ref{cor maximal HA <==> maximal SZ}).

\begin{theorem}
\label{thm ID ==> (HA<==>SZ)}If $\left(  \mathcal{C},\mathcal{K}\right)  $ is
a connective island domain, then a subfamily of $\mathcal{C}$ is a system of
pre-islands if and only if it is admissible.
\end{theorem}

\begin{proof}
We have already seen in Proposition~\ref{prop SZ==>HA} that every system of
pre-islands is admissible. Let us now assume that $\left(  \mathcal{C}%
,\mathcal{K}\right)  $ is a connective island domain and let $\mathcal{H}%
\subseteq\mathcal{C}$ be admissible. From\ Lemma~\ref{lemma ID ==> (HA==>CDW)}
it follows that $\mathcal{H}$ is $\operatorname*{CDW}$-independent. Let
$\mathcal{S}$ be the system of pre-islands corresponding to $\left(
\mathcal{C},\mathcal{K},h_{\mathcal{H}}\right)  $, where $h_{\mathcal{H}}$ is
the canonical height function of $\mathcal{H}$ (see
Definition~\ref{def canonical height function}). Then $\mathcal{S}$ is also
$\operatorname*{CDW}$-independent by Theorem~\ref{thm ID <==> (SZ==>CD)}. From
Proposition~\ref{prop HA==>SZ resze} it follows that $\mathcal{H}%
\subseteq\mathcal{S}$, and we are going to prove that we actually have
$\mathcal{H}=\mathcal{S}$.

Suppose for contradiction that there exists $S\in\mathcal{S}$ such that
$S\notin\mathcal{H}$. Since $\mathcal{H}$ is $\operatorname{CD}$-independent
and finite, the members of $\mathcal{H}$ that contain $S$ form a nonempty
finite chain. Denoting the least element of this chain by $H$, we have
$S\subset H$, as $S\notin\mathcal{H}$. Let $M_{1},\ldots,M_{m}$ denote those
elements of $\mathcal{H}$ that are properly contained in $H$ and are maximal
with respect to this property (if there are such sets). Clearly, $M_{1}%
,\ldots,M_{m}$ are pairwise disjoint, and $M_{1}\cup\cdots\cup M_{m}\subset
H$, since $\mathcal{H}$ is $\operatorname*{CDW}$-independent (see
Remark~\ref{rem CD maximal subsets}).

We claim that $S\nsubseteq M_{1}\cup\cdots\cup M_{m}$. Assuming on the
contrary that $S\subseteq M_{1}\cup\cdots\cup M_{m}$, the $\operatorname*{CDW}%
$-independence of $\mathcal{S}$ implies that there is an $i\in\left\{
1,\ldots,m\right\}  $ such that $S\subseteq M_{i}$. However, this contradicts
the minimality of $H$.

Any two elements of $H\setminus\left(  M_{1}\cup\cdots\cup M_{m}\right)  $ are
contained in exactly the same members of $\mathcal{H}$, therefore
$h_{\mathcal{H}}$ is constant, say constant $c$, on this set (see
Figure~\ref{fig proof}; cf. also Figure~\ref{fig kanonikus}). On the other
hand, if $x\in M_{1}\cup\cdots\cup M_{m}$, then clearly we have
$h_{\mathcal{H}}\left(  x\right)  \geq c$, hence $\min h_{\mathcal{H}}\left(
H\right)  =c$. Since $S$ is not covered by the sets $M_{i}$, it contains a
point $u$ from $H\setminus\left(  M_{1}\cup\cdots\cup M_{m}\right)  $,
therefore $\min h_{\mathcal{H}}\left(  S\right)  =h\left(  u\right)  =c$. Thus
we have $S\subset H\in\mathcal{K}$ and $\min h_{\mathcal{H}}\left(  S\right)
=\min h_{\mathcal{H}}\left(  H\right)  $, contradicting that $S$ is a
pre-island with respect to $\left(  \mathcal{C},\mathcal{K},h_{\mathcal{H}%
}\right)  $.
\end{proof}

%

\begin{figure}
[tb]
\begin{center}
\includegraphics[
natheight=5.446500in,
natwidth=7.871500in,
height=4.2514cm,
width=6.1162cm
]%
{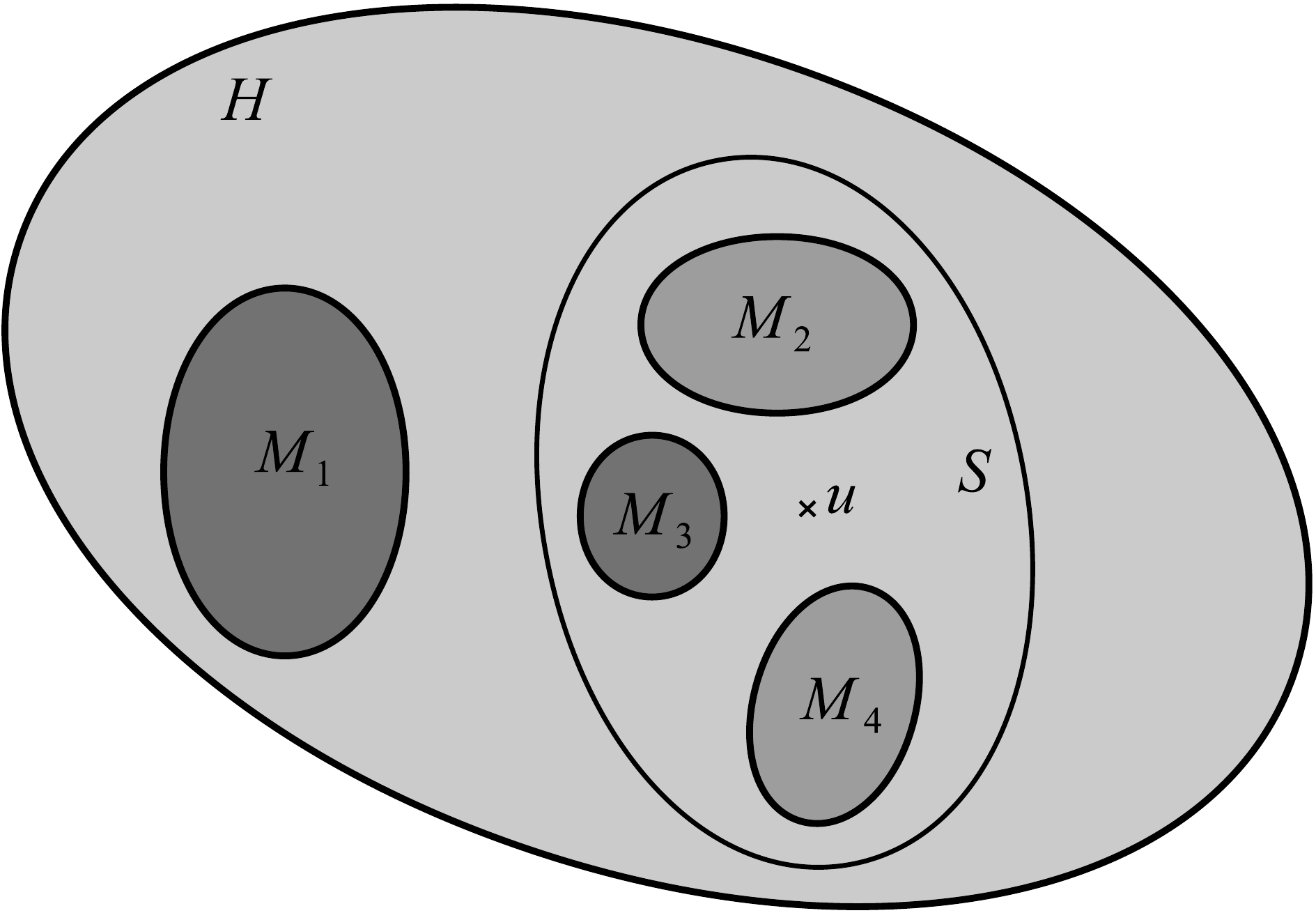}%
\caption{Illustration to the proof of Theorem~\ref{thm ID ==> (HA<==>SZ)}}%
\label{fig proof}%
\end{center}
\end{figure}

The maximum number of (pre-)islands certainly depends on the structure of the
island domain $\left(  \mathcal{C},\mathcal{K}\right)  $. H\"{a}rtel
\cite{Hartel} proved that the maximum number of rectangular islands on a
$1\times n$ board is $n$, and Cz\'{e}dli \cite{czedli} generalized this result
by showing that the maximum number of rectangular islands on an $n\times m$
board is $\left\lfloor \left(  mn+m+n-1\right)  /2\right\rfloor $. Although
these are the only cases where the exact value is known, there are estimates
in several other cases \cite{BHH,HNP,HHNS,L3,P}. In full generality, we have
the following upper bound.

\begin{theorem}
\label{thm maximum<=|U|}If $\left(  \mathcal{C},\mathcal{K}\right)  $ is a
connective island domain and $\mathcal{S}$ is a system of pre-islands
corresponding to $\left(  \mathcal{C},\mathcal{K}\right)  $, then $\left\vert
S\right\vert \leq\left\vert U\right\vert $.
\end{theorem}

\begin{proof}
Let $\left(  \mathcal{C},\mathcal{K}\right)  $ be a connective island domain
and let $\mathcal{S}\subseteq\mathcal{C}\setminus\left\{  \emptyset\right\}  $
be a system of pre-islands corresponding to $\left(  \mathcal{C}%
,\mathcal{K}\right)  $. By Theorem~\ref{thm ID <==> (SZ==>CD)}, $\mathcal{S}$
is $\operatorname*{CDW}$-independent, and hence $\mathcal{S}\cup\left\{
\emptyset\right\}  $ is also $\operatorname*{CDW}$-independent. From the
results of \cite{CzSch} it follows that every maximal $\operatorname*{CDW}%
$-independent subset of $\mathcal{P}\left(  U\right)  $ has $\left\vert
U\right\vert +1$ elements. Thus we have $\left\vert \mathcal{S}\right\vert
+1\leq\left\vert U\right\vert +1$.
\end{proof}

Observe that the above mentioned result of H\"{a}rtel shows that the bound
obtained in Theorem~\ref{thm maximum<=|U|} is sharp.

\section{Islands and proximity domains\label{section strict}}

In this section we investigate islands, and we give a characterization of
systems of islands corresponding to island domains $\left(  \mathcal{C}%
,\mathcal{K}\right)  $ satisfying certain natural conditions. We define a
binary relation $\delta\subseteq\mathcal{C}\times\mathcal{C}$ that expresses
the fact that a set $B\in\mathcal{C}$ is in some sense close to a set
$A\in\mathcal{C}$:%
\begin{equation}
A\delta B\Leftrightarrow\exists K\in\mathcal{K}:~\text{ }A\preceq K\text{ and
}K\cap B\neq\emptyset. \label{eq def delta}%
\end{equation}

\begin{remark}
Let us note that the relation $\delta$ is not always symmetric. As an example,
consider a directed graph, and let $\mathcal{C}=\mathcal{K}$ consist of $U$ and of those
sets $S$ of vertices that have a source. (By a source of a set $S$ we mean a
vertex $s\in S$ from which all other vertices of $S$ can be reached by a
directed path that lies entirely in $S$.) It is easy to verify that in the
graph $a\rightarrow b\rightarrow c\leftarrow d\leftarrow e$ we have $A\delta
B$ but not $B\delta A$ for the sets $A=\left\{  a,b\right\}  $ and $B=\left\{
c,d\right\}  $.
\end{remark}

\begin{definition}
\label{def distant}We say that $A,B\in\mathcal{C}$ are \emph{distant} if
neither $A\delta B$ nor $B\delta A$ holds. Obviously, in this case
$A$ and $B$ are also incomparable (in fact, disjoint), whenever $A,B\neq
\emptyset$. A nonempty family $\mathcal{H}\subseteq\mathcal{C}$ will be called
a \emph{distant family}, if any two incomparable members of $\mathcal{H}$ are distant.
\end{definition}

\begin{remark}
\label{rem our delta is almost proximity}It is not difficult to verify that relation
$\delta$ satisfies the following properties for all $A,B,C\in\mathcal{C}$ whenever $B\cup C \in \mathcal{C}: $ %
\begin{align*}
A\delta B  &  \Rightarrow B\neq\emptyset;\\
A\cap B\neq\emptyset &  \Rightarrow A\delta B;\\
A\delta(B\cup C)  &  \Leftrightarrow(A\delta B\text{ or }A\delta C).
\end{align*}

\end{remark}

\begin{lemma}
\label{lemma distant ==> (CDW & HA)}If $\mathcal{H\subseteq C}$ is a distant
family, then $\mathcal{H}$ is $\operatorname*{CDW}$-independent. Moreover, if
$U\in\mathcal{H}$, then $\mathcal{H}$ is admissible.
\end{lemma}

\begin{proof}
Let $\mathcal{H}\subseteq\mathcal{C}$ be a distant family, then $\mathcal{H}$
is clearly $\operatorname{CD}$-independent; moreover, it is easy to show using
Remark~\ref{rem CD maximal subsets} that $\mathcal{H}$ is $\operatorname*{CDW}%
$-independent.

Next let us assume that $U\in\mathcal{H}$; we shall prove that $\mathcal{H}$
is admissible. Let $\mathcal{A}\subseteq\mathcal{H}$ be an antichain and let
$H\in\mathcal{A}$. If $K\in\mathcal{K}$ contains $H$ properly, then there is a
cover $K_{1}\in\mathcal{K}$ of $H$ such that $H\prec K_{1}\subseteq K$. Since
all members of $\mathcal{A}\setminus\left\{  H\right\}  $ are distant from
$H$, none of them can intersect $K_{1}$, and therefore we have $K_{1}%
\nsubseteq\bigcup\,\mathcal{A}$, and hence $K\nsubseteq\bigcup\,\mathcal{A}$.
\end{proof}

\begin{remark}
\label{rem distant ==> standard h}Note that we have proved that $\mathcal{H}$
satisfies (\ref{eq stronger admissible}) for every antichain
$\mathcal{A\subseteq H}$. Thus $h_{\mathcal{H}}$ is the standard height
function of $\mathcal{H}$.
\end{remark}

\begin{theorem}
\label{thm ID ==> (distant==>strSZ)}Let $\left(  \mathcal{C},\mathcal{K}%
\right)  $ be a connective island domain and let $\mathcal{H}\subseteq
\mathcal{C}\setminus\left\{  \emptyset\right\}  $ with $U\in\mathcal{H}$. If
$\mathcal{H}$ is a distant family, then $\mathcal{H}$ is a system of islands;
moreover, $\mathcal{H}$ is the system of islands corresponding to its standard
height function.
\end{theorem}

\begin{proof}
Let $\mathcal{H}\subseteq\mathcal{C}\setminus\left\{  \emptyset\right\}  $ be
a distant family such that $U\in\mathcal{H}$. Applying
Lemma~\ref{lemma distant ==> (CDW & HA)} we obtain that $\mathcal{H}$ is
admissible, hence $\mathcal{H}$ is the system of pre-islands corresponding to
$\left(  \mathcal{C},\mathcal{K},h_{\mathcal{H}}\right)  $ by
Theorem~\ref{thm ID ==> (HA<==>SZ)}. Moreover, $h_{\mathcal{H}}$ is the
standard height function of $\mathcal{H}$ by
Remark~\ref{rem distant ==> standard h}.

To finish the proof, we will prove that each $H\in\mathcal{H}$ is actually an
island with respect to $\left(  \mathcal{C},\mathcal{K},h_{\mathcal{H}%
}\right)  $. Suppose that $K\in\mathcal{K}$ is a cover of $H$. The distantness
of $\mathcal{H}$ implies that the only members of $\mathcal{H}$ that intersect
$K\setminus H$ are the ones that properly contain $H$. Since $h_{\mathcal{H}}$
is the standard height function, $h_{\mathcal{H}}\left(  u\right)  <\min
h_{\mathcal{H}}\left(  H\right)  $ follows for all $u\in K\setminus H$.
\end{proof}

\begin{definition}
\label{def proximity domain}The island domain $\left(  \mathcal{C}%
,\mathcal{K}\right)  $ is called a \emph{proximity domain}, if it is a
connective island domain and the relation $\delta$ is symmetric for nonempty
sets, that is
\begin{equation}
\forall A,B\in\mathcal{C}\setminus\left\{  \emptyset\right\}  :~A\delta
B\Leftrightarrow B\delta A. \label{eq delta symmetric}%
\end{equation}

\end{definition}

If a relation $\delta$ defined on $\mathcal{P}\left(  U\right)  $ satisfies
the three properties of Remark~\ref{rem our delta is almost proximity} and
$\delta$ is symmetric for nonempty sets, then $\left(  U,\delta\right)  $ is
called a \emph{proximity space}. The notion apparently goes back to Frigyes
Riesz \cite{Riesz}, however this axiomatization is due to Vadim~A.~Efremovich
(see \cite{E}).

\begin{proposition}
\label{prop PD ==> (strSZ==>distant)}If $\left(  \mathcal{C},\mathcal{K}%
\right)  $ is a proximity domain, then any system of islands corresponding to
$\left(  \mathcal{C},\mathcal{K}\right)  $ is a distant system.
\end{proposition}

\begin{proof}
Let $\left(  \mathcal{C},\mathcal{K}\right)  $ be a proximity domain, and let
$\mathcal{S}$ be the system of islands corresponding to $\left(
\mathcal{C},\mathcal{K},h\right)  $ for some height function $h$. Since
$\left(  \mathcal{C},\mathcal{K}\right)  $ is a connective island domain,
$\mathcal{S}$ is $\operatorname{CD}$-independent according to
Theorem~\ref{thm ID <==> (SZ==>CD)}. Therefore, if $A,B\in\mathcal{S}$ are
incomparable, then we have $A\cap B=\emptyset$. Assume for contradiction that
$A\delta B$, i.e. that there is a set $K\in\mathcal{K}$ such that $A\prec K$
and $B\cap K\neq\emptyset$. Since $A$ and $B$ are disjoint, there exists an
element $b\in\left(  B\cap K\right)  \setminus A$. Similarly, as we have
$B\delta A$ by (\ref{eq delta symmetric}), there exists an element
$a\in\left(  A\cap K^{\prime}\right)  \setminus B$ for some $K^{\prime}%
\in\mathcal{K}$ with $B\prec K^{\prime}$. By making use of the fact that both $A$
and $B$ are islands with respect to $\left(  \mathcal{C},\mathcal{K},h\right)
$, we obtain the following contradicting inequalities:%
\begin{align*}
h\left(  b\right)   &  <\min h\left(  A\right)  \leq h\left(  a\right)  ;\\
h\left(  a\right)   &  <\min h\left(  B\right)  \leq h\left(  b\right)  .%
\qedhere
\end{align*}

\end{proof}

From Theorem~\ref{thm ID ==> (distant==>strSZ)} and
Proposition~\ref{prop PD ==> (strSZ==>distant)} we obtain immediately the
following characterization of systems of islands for proximity domains.

\begin{corollary}
\label{cor PD ==> (distant<==>strSZ)}If $\left(  \mathcal{C},\mathcal{K}%
\right)  $ is a proximity domain, and $\mathcal{H}\subseteq\mathcal{C}%
\setminus\left\{  \emptyset\right\}  $ with $U\in\mathcal{H}$, then
$\mathcal{H}$ is a system of islands if and only if $\mathcal{H}$ is a distant
family. Moreover, in this case $\mathcal{H}$ is the system of islands
corresponding to its standard height function.
\end{corollary}

Finally, let us consider the following condition on $\left(  \mathcal{C}%
,\mathcal{K}\right)  $, which is stronger than that of being a connective
island domain:%
\begin{equation}
\forall K_{1},K_{2}\in\mathcal{K}:~K_{1}\cap K_{2}\neq\emptyset\implies
K_{1}\cup K_{2}\in\mathcal{K}. \label{eq KK}%
\end{equation}
Observe that if we have a graph structure on $U$, and $\left(  \mathcal{C}%
,\mathcal{K}\right)  $ is a corresponding island domain (cf.
Example~\ref{example graphs}), then (\ref{eq KK}) holds.

\begin{theorem}
\label{thm KK ==> proximity}Suppose that $\left(  \mathcal{C},\mathcal{K}%
\right)  $ satisfies condition $\left(  \ref{eq KK}\right)  $, and assume that
for all $C\in\mathcal{C}$, $K\in\mathcal{K}$ with $C\prec K$ we have
$\left\vert K\setminus C\right\vert =1$. Then $\left(  \mathcal{C}%
,\mathcal{K}\right)  $ is a proximity domain, and pre-islands and islands
corresponding to $\left(  \mathcal{C},\mathcal{K}\right)  $ coincide.
Therefore, if $\mathcal{H}\subseteq\mathcal{C}\setminus\left\{  \emptyset
\right\}  $ and $U\in\mathcal{H}$, then $\mathcal{H}$ is a system of
(pre-)islands if and only if $\mathcal{H}$ is a distant family. Moreover, in
this case $\mathcal{H}$ is the system of (pre-)islands corresponding to its
standard height function.
\end{theorem}

\begin{proof}
Let $A,B\in\mathcal{C}\setminus\left\{  \emptyset\right\}  $ such that
$A\delta B$, i.e. $K\cap B\neq\emptyset$ for some $K\in\mathcal{K}$ with
$A\preceq K$. If $A\cap B\neq\emptyset$, then clearly $B\delta A$ holds.
Suppose now that $A\cap B=\emptyset$. By our assumption, $K=A\cup\left\{
b\right\}  $ for some $b\in B$. From (\ref{eq KK}) it follows that $K\cup
B\in\mathcal{K}$. Since $B\subset A\cup B=K\cup B\in\mathcal{K}$, there exists
a cover $K^{\prime}\in\mathcal{K}$ of $B$ such that $B\prec K^{\prime
}\subseteq A\cup B$. Clearly, we have $K^{\prime}\cap A\neq\emptyset$, hence
$B\delta A$, and this proves that the relation $\delta$ is symmetric.
Condition (\ref{eq KK}) is stronger than (\ref{eq def ID}), therefore $\left(
\mathcal{C},\mathcal{K}\right)  $ is a proximity domain.

From our assumptions it is trivial that every pre-island with respect to
$\left(  \mathcal{C},\mathcal{K}\right)  $ is also an island. The last two
statements follow then from Corollary~\ref{cor PD ==> (distant<==>strSZ)}.
\end{proof}

\begin{corollary}
Let $G$ be a graph with vertex set $U$; let $\left(  \mathcal{C}%
,\mathcal{K}\right)  $ be an island domain corresponding to $G$ (cf.
Example~\ref{example graphs}), and let $\mathcal{H}\subseteq\mathcal{C}%
\setminus\left\{  \emptyset\right\}  $ with $U\in\mathcal{H}$. Then
$\mathcal{H}$ is a system of (pre-)islands if and only if $\mathcal{H}$ is
distant; moreover, in this case $\mathcal{H}$ is the system of (pre-)islands
corresponding to its standard height function.
\end{corollary}

\section{Concluding remarks and an alternative framework}

We introduced the notion of a (pre-)island corresponding to an island domain
$\left(  \mathcal{C},\mathcal{K}\right)  $, where $U\in\mathcal{C}%
\subseteq\mathcal{K}\subseteq\mathcal{P}\left(  U\right)  $ for a nonempty
finite set $U$. We described island domains $\left(  \mathcal{C}%
,\mathcal{K}\right)  $ having $\operatorname{CD}$-independent systems of
pre-islands, and we characterized systems of (pre-)islands for such island
domains. In the general case, when no assumption is made on $\left(
\mathcal{C},\mathcal{K}\right)  $, we gave a necessary condition for a family
of sets to be a system of pre-islands, and it remains an open problem to find
an appropriate necessary and sufficient condition. Nevertheless, we obtained a
complete characterization of \emph{maximal} systems of pre-islands in this
general case. Determining the size of these maximal systems of pre-islands for
specific island domains $\left(  \mathcal{C},\mathcal{K}\right)  $ has been,
and continues to be, a topic of active research.

Before concluding the paper, let us propose another possible approach to
define islands. Let $U$ be a nonempty finite set and let $\mathcal{C}%
\subseteq\mathcal{P}\left(  U\right)  $ with $U\in\mathcal{C}$, as before. We
describe the \textquotedblleft surroundings\textquotedblright\ of members of
$\mathcal{C}$ by means of a relation $\eta\subseteq U\times\mathcal{C}$, where
$u\eta C$ means that the point $u\in U$ is close to the set $C\in\mathcal{C}$.
We require $\eta$ to satisfy the following very natural axiom:%
\begin{equation}
\forall u\in U~\forall C\in\mathcal{C}:~u\in C\implies u\eta C.
\label{eq eta axiom}%
\end{equation}
Examples of such \textquotedblleft point-to-set\textquotedblright\ proximity
relations include closure systems (in particular, topological spaces) with
$u\eta C$ if and only if $u$ belongs to the closure of $C$, and graphs with
$u\eta C$ if and only if $u$ belongs to the neighborhood of $C$. We shall call
a pair $\left(  \mathcal{C},\eta\right)  $ satisfying (\ref{eq eta axiom}) an
\emph{island domain}.

For any $C\in\mathcal{C}$, the set $\partial C:=\left\{  u\in U\colon u\eta
C\text{ and }u\notin C\right\}  $ is the set of points that surround $C$ (note
that this is \emph{not} the usual notion of boundary for topological spaces).
Therefore, we define islands corresponding to $\left(  \mathcal{C}%
,\eta\right)  $ as follows: If $h\colon U\rightarrow\mathbb{R}$ is a height
function and $S\in\mathcal{C}$, then we say that $S$ is an \emph{island with
respect to }$\left(  \mathcal{C},\eta,h\right)  $, if $h\left(  u\right)
<\min h\left(  S\right)  $ holds for all $u\in\partial S$. This definition is
similar in spirit to the definition of an island corresponding to an island
domain $\left(  \mathcal{C},\mathcal{K}\right)  $; in fact, it is a
generalization of it. To see this, let us consider a pair $\left(
\mathcal{C},\mathcal{K}\right)  $, and let us define $\eta\subseteq
U\times\mathcal{C}$ as follows:%
\[
u\eta C\iff\exists K\in\mathcal{K}:C\preceq K\text{ and }u\in K.
\]
It is easy to verify that the islands corresponding to $\left(  \mathcal{C}%
,\eta\right)  $ are exactly the islands corresponding to $\left(
\mathcal{C},\mathcal{K}\right)  $.

Let us now briefly sketch how to adapt the definitions of admissibility,
connective island domain and distantness to this setting. We shall say that
$\mathcal{H}\subseteq\mathcal{C}\setminus\left\{  \emptyset\right\}  $ is
\emph{admissible}, if $U\in\mathcal{H}$, and for every antichain
$\mathcal{A}\subseteq\mathcal{H}$ we have%
\[
\exists H\in\mathcal{A}\text{ such that }\forall u\in U:~u\in\partial
H\implies u\notin\bigcup\,\mathcal{A}.
\]
We call the pair $\left(  \mathcal{C},\eta\right)  $ a \emph{connective island
domain} if%
\[
\forall A,B\in\mathcal{C}:~\left(  A\cap B\neq\emptyset\text{ and }B\nsubseteq
A\right)  \implies\exists u\in B\setminus A:u\eta A.
\]
To define distantness, we extend $\eta$ to a \textquotedblleft
set-to-set\textquotedblright\ proximity relation $\delta\subseteq
\mathcal{C}\times\mathcal{C}$: for $A,B\in\mathcal{C}$, let $A\delta B$ if and
only if there exists a point $u\in B$ with $u\eta A$. Using this relation
$\delta$, we can define distant families just as in
Definition~\ref{def distant}.

Most of the results of this paper remain valid with these new definitions, and
the proofs require only minor and quite straightforward modifications. The
only exceptions are Lemma~\ref{lemma distant ==> (CDW & HA)}, where we need
the extra assumption that $\left(  \mathcal{C},\eta\right)  $ is a connective
island domain, and Theorem~\ref{thm KK ==> proximity}, which cannot be
interpreted in this framework, as it refers to $\mathcal{K}$. The following
theorem summarizes the main results.

\begin{theorem}
\label{thm eta}Let $U$ be a nonempty finite set, let $\mathcal{C}%
\subseteq\mathcal{P}\left(  U\right)  $ with $U\in\mathcal{C}$, and let
$\eta\subseteq U\times\mathcal{C}$ satisfy $\left(  \ref{eq eta axiom}\right)
$.%
\renewcommand{\theenumi}{(\roman{enumi})}
\renewcommand{\labelenumi}{\theenumi}%

\begin{enumerate}
\item A family $\mathcal{H}\subseteq\mathcal{C}\setminus\left\{
\emptyset\right\}  $ is contained in a system of islands if and only if
$\mathcal{H}$ is admissible.

\item A family $\mathcal{H}\subseteq\mathcal{C}\setminus\left\{
\emptyset\right\}  $ is a maximal system of islands if and only if
$\mathcal{H}$ is a maximal admissible family.

\item The pair $\left(  \mathcal{C},\eta\right)  $ is a connective island
domain if and only if all systems of islands are $\operatorname{CD}%
$-independent (equivalently, $\operatorname*{CDW}$-independent).

\item If $\left(  \mathcal{C},\eta\right)  $ is a connective island domain,
then a family $\mathcal{H}\subseteq\mathcal{C}\setminus\left\{  \emptyset
\right\}  $ is a system of islands if and only if $\mathcal{H}$ is admissible.

\item If $\left(  \mathcal{C},\eta\right)  $ is a connective island domain and
the corresponding relation $\delta$ is symmetric, then a family $\mathcal{H}%
\subseteq\mathcal{C}\setminus\left\{  \emptyset\right\}  $ is a system of
islands if and only if $\mathcal{H}$ is distant and $U\in\mathcal{H}$.
Moreover, in this case $\mathcal{H}$ is the system of islands corresponding to
its standard height function.
\end{enumerate}
\end{theorem}

\begin{corollary}
\label{cor eta}Let $G=\left(  U,E\right)  $ be a connected simple graph, let
$\mathcal{C}\subseteq\mathcal{P}\left(  U\right)  $ be a family of connected
subsets with $U\in\mathcal{C}$, and let us define $\eta\subseteq
U\times\mathcal{C}$ by%
\[
u\eta C\iff u\in C\text{ or }\exists v\in C:~uv\in E.
\]
Then the following three conditions are equivalent for any $\mathcal{H}%
\subseteq\mathcal{C}\setminus\left\{  \emptyset\right\}  $ with $U\in
\mathcal{H}$:%
\renewcommand{\theenumi}{(\roman{enumi})}
\renewcommand{\labelenumi}{\theenumi}%

\begin{enumerate}
\item $\mathcal{H}$ is a system of islands corresponding to $\left(
\mathcal{C},\eta\right)  $.

\item $\mathcal{H}$ is an admissibly family.

\item $\mathcal{H}$ is a distant family.
\end{enumerate}

\noindent If these conditions hold, then $\mathcal{H}$ is the system of
islands corresponding to its standard height function.
\end{corollary}

\begin{proof}
The fact that $\mathcal{C}$ contains only connected sets ensures that $\left(
\mathcal{C},\eta\right)  $ is a connective island domain, and it is trivial
that $\delta$ is symmetric, hence we can apply Theorem~\ref{thm eta}.
\end{proof}

Let us note that in Corollary~\ref{cor eta} distantness of two sets
$A,B\in\mathcal{C}$ means that there is no edge with one endpoint in $A$ and
the other endpoint in $B$. Applying this corollary to a square grid (on a
rectangular, cylindrical or toroidal board) or to a triangular grid, and
letting $\mathcal{C}$ consist of all rectangles, squares or triangles, we
obtain the earlier dry characterizations of islands as special cases.

\subsection*{Acknowledgments}

S\'{a}ndor Radeleczki acknowledges that this research was carried out as part
of the TAMOP-4.2.1.B-10/2/KONV-2010-0001 project supported by the European
Union, co-financed by the European Social Fund.

Eszter K. Horv\'{a}th and Tam\'{a}s Waldhauser acknowledge the support 
of the Hungarian National
Foundation for Scientific Research under grant no. K83219.
Supported by the European Union and co-funded by the European Social Fund
   under the project ``Telemedicine-focused research activities on the field of Matematics,
   Informatics and Medical sciences'' of project number
     ``T\'AMOP-4.2.2.A-11/1/KONV-2012-0073''

Stephan Foldes acknowledges that this work has been co-funded by Marie Curie
Actions and supported by the National Development Agency (NDA) of Hungary and
the Hungarian Scientific Research Fund (OTKA, contract number 84593), within a project hosted by the
University of Miskolc, Department of Analysis. The work was also completed as
part of the TAMOP-4.2.1.B.- 10/2/KONV-2010-0001 project at the University of
Miskolc, with support from the European Union, co-financed by the European
Social Fund.

\bigskip

\bigskip

\includegraphics[height=0.9cm]{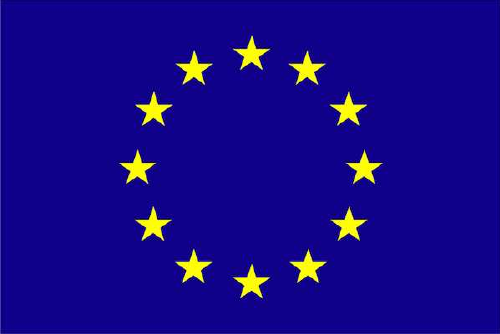} \quad
\includegraphics[height=1.2cm]{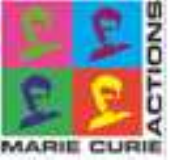} \quad
\includegraphics[height=1.2cm]{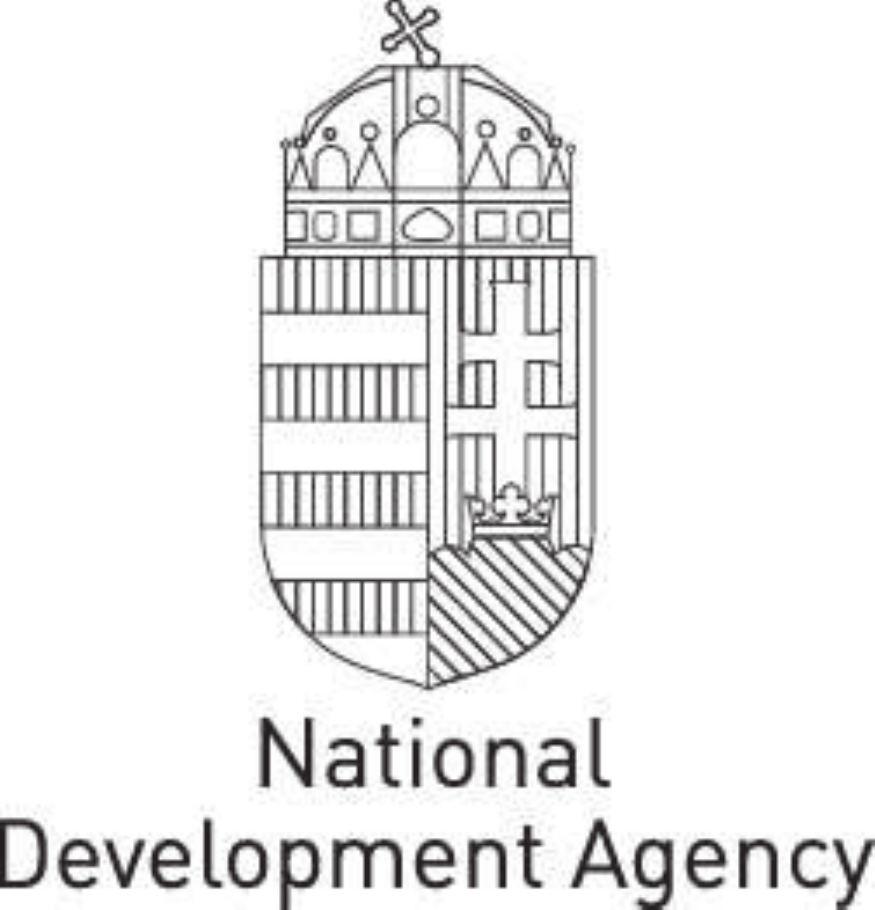} \quad\includegraphics[height=1cm]{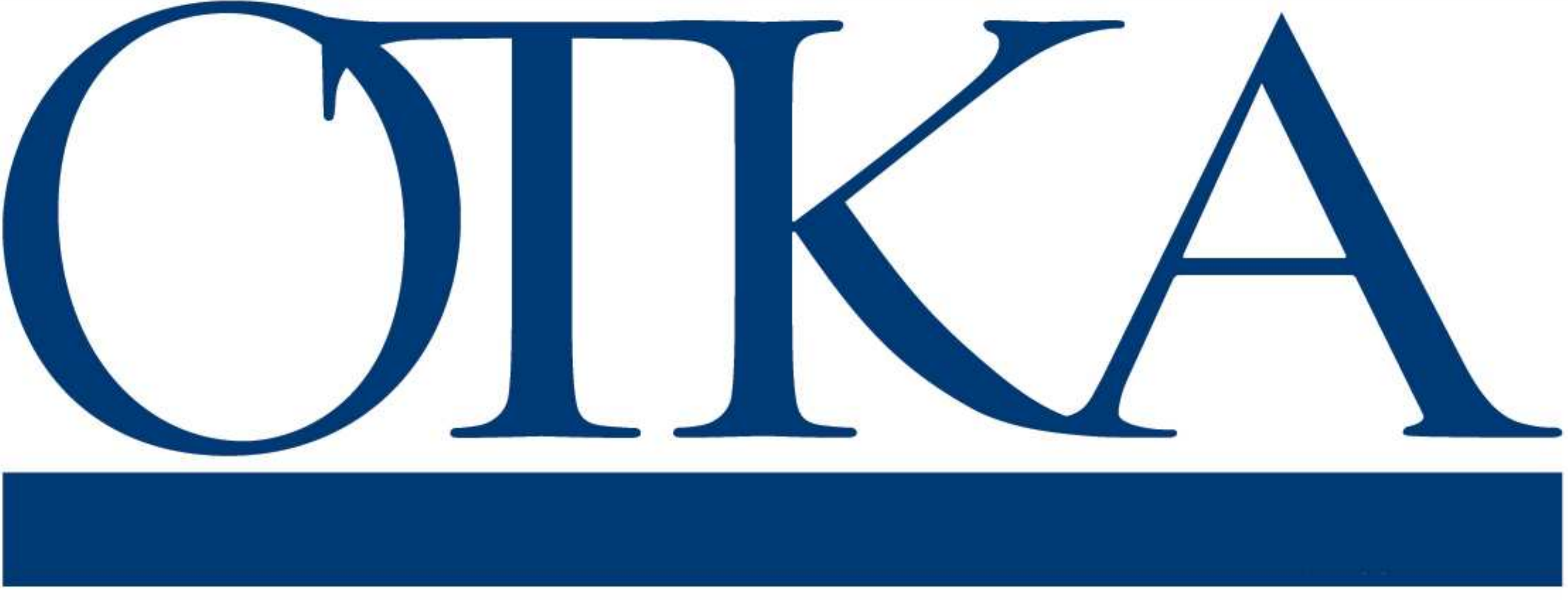}


\bigskip

\begin{thebibliography}{99}                                                                                               %


\bibitem {BHH}J. Bar\'{a}t, P. Hajnal and E. K. Horv\'{a}th, Elementary proof
techniques for the maximum number of islands, European J. Combin. 32 (2011), 276--281.

\bibitem {CrHa}Y. Crama and P.L. Hammer, Boolean functions. Theory,
algorithms, and applications. Encyclopedia of Mathematics and its Applications
142, Cambridge University Press 2011.

\bibitem {czedli}G. Cz\'{e}dli, The number of rectangular islands by means of
distributive lattices, European J. Combin. 30 (2009), 208--215.

\bibitem {CzHaSch}G. Cz\'{e}dli, M. Hartmann and E. T. Schmidt, CD-independent
subsets in distributive lattices, Publ. Math. Debrecen 74 (2009), 127--134.

\bibitem {CzHSch}G.~Cz\'{e}dli, A.~P.~Huhn and E.~T.~Schmidt, Weakly
independent subsets in lattices, Algebra Universalis 20 (1985), 194--196.

\bibitem {CzSch}G. Cz\'{e}dli and E. T. Schmidt, CDW-independent subsets in
distributive lattices, Acta Sci. Math. (Szeged) 75 (2009), 49--53.

\bibitem {E}V. A. Efremovich, Infinitesimal spaces, Dokl. Akad. Nauk SSSR 76
(1951), 341--343 (In Russian).

\bibitem {FH}S.\ Foldes and P.\ L.\ Hammer, Disjunctive and conjunctive
representations in finite lattices and convexity spaces, Discrete Math. 31
(2006), 307--316.

\bibitem {FS}S.\ Foldes and N.\ M.\ Singhi, On instantaneous codes, J. Comb.
Inf. Syst. Sci. 258 (2002), 13--25.

\bibitem {GW}B. Ganter and R. Wille, Formal Concept Analysis: Mathematical
Foundations, Springer Verlag 1998.

\bibitem {Hartel}G. H\"{a}rtel, personal communication (2007).

\bibitem {HR}E. K. Horv\'{a}th and S. Radeleczki, Notes on CD-independent
subsets, Acta Sci. Math. (Szeged) 78 (2012), 3--24.

\bibitem {HTM}E. K. Horv\'{a}th, A. M\'{a}der and A.\ Tepav\v{c}evi\'{c}:
One-dimensional Cz\'{e}dli-type islands, College Math. J. 42 (2011), 374--378.

\bibitem {HNP}E. K. Horv\'{a}th, Z. N\'{e}meth and G. Pluh\'{a}r, The number
of triangular islands on a triangular grid, Period. Math. Hungar. 58 (2009), 25--34.

\bibitem {HHNS}E. K. Horv\'{a}th, G. Horv\'{a}th, Z. N\'{e}meth and Cs.
Szab\'{o}, The number of square islands on a rectangular sea, Acta Sci. Math.
(Szeged) 76 (2010), 35--48.

\bibitem {HST}E. K. Horv\'{a}th, B.\ \v{S}e\v{s}elja and
A.\ Tepav\v{c}evi\'{c}, Cut approach to islands in rectangular fuzzy
relations, Fuzzy Sets and Systems 161 (2010), 3114--3126.

\bibitem {HST2}E. K. Horv\'{a}th, B.\ \v{S}e\v{s}elja and
A.\ Tepav\v{c}evi\'{c}, Cardinality of height function's range in case of
maximally many rectangular islands -- computed by cuts, to appear in Cent.
Eur. J. Math.

\bibitem {L}Zs. Lengv\'{a}rszky, The minimum cardinality of maximal systems of
rectangular islands, European J. Combin. 30 (2009), 216--219.

\bibitem {L2}Zs.\ Lengv\'{a}rszky, Notes on triangular islands, Acta Sci.
Math. 75 (2009), 369--376.

\bibitem {L3}Zs.\ Lengv\'{a}rszky, The size of maximal systems of square
islands, European J. Combin. 30 (2009), 889--892.

\bibitem {LP}Zs.\ Lengv\'{a}rszky and P. P. Pach, A note on rectangular
islands: the continuous case, Acta Sci. Math. (Szeged) 77 (2011), 27--34.

\bibitem {MM}A. M\'{a}der and G. Makay, The maximum number of rectangular
islands, The Teaching of Mathematics 1461 (2011), 31--44.

\bibitem {MV}A. M\'{a}der and R. Vajda, Elementary Approaches to the Teaching
of the Combinatorial Problem of Rectangular Islands, Int. J. Comput. Math.
Learn. 15 (2010), 267--281.

\bibitem {P}G. Pluh\'{a}r, The number of brick islands by means of
distributive lattices, Acta Sci. Math. (Szeged) 75 (2009), 3--11.

\bibitem {PPPSz}P. P. Pach, G. Pluh\'{a}r, A. Pongr\'{a}cz and Cs. Szab\'{o},
The possible number of islands on the sea, J. Math. Anal. Appl. 375 (2011), 8--13.

\bibitem {Riesz}F. Riesz, Stetigkeitsbegriff und abstrakte Mengenlehre, Atti
del IV Congresso Internazionale dei Matematici II (1908), 18--24.
\end{thebibliography}
\end{document}